\documentclass[12pt]{amsart}
\usepackage[colorlinks=true,citecolor=black,linkcolor=black,urlcolor=blue]{hyperref}
\usepackage{amsmath}
\usepackage[english,  activeacute]{babel}
\usepackage[utf8]{inputenc}
\usepackage{amssymb}
\usepackage{amsthm}
\usepackage{graphics,graphicx}
\usepackage{enumerate}
\usepackage{array}
\usepackage{bm}
\usepackage{a4wide}
\usepackage{color, url}
\usepackage{float}

\theoremstyle{plain}
\newtheorem{theorem}{Theorem}[section]
\newtheorem{proposition}[theorem]{Proposition}
\newtheorem{corollary}[theorem]{Corollary}

\theoremstyle{definition}

\newcommand{\Sp}{{\mathcal{S}_{\sf{Pal}}}}

\title{Palindromic and Colored Superdiagonal Compositions}

\author{Jazm\'in Mantilla}
\address{\noindent Escuela de Matem\'aticas,  Universidad Industrial de Santander, Bucaramanga,  COLOMBIA}
\email{jazlismaro@hotmail.com}

\author{Wilson Olaya-León}
\address{\noindent Escuela de Matem\'aticas, Universidad Industrial de Santander, Bucaramanga,  COLOMBIA}
\email{wolaya@uis.edu.co}

\author{Jos\'e L. Ram\'{\i}rez}
\address{\noindent Departamento de Matem\'aticas, Universidad Nacional de Colombia, Bogot\'a,  COLOMBIA}
\email{jlramirezr@unal.edu.co}
\urladdr{http://sites.google.com/site/ramirezrjl}

\date{\today}

\begin{document}
\begin{abstract}
A superdiagonal composition is one in which the $i$-th part or summand is of size greater than or equal to $i$.  In this paper, we study the number of   palindromic superdiagonal compositions and colored superdiagonal compositions. In particular, we give generating functions and explicit combinatorial formulas involving binomial coefficients and Stirling numbers of the first kind. \end{abstract}

\maketitle

\section{Introduction and Notation}
A \emph{composition} of a positive integer $n$ is a sequence of positive integers $\sigma=(\sigma_1, \sigma_2,\ldots, \sigma_\ell)$ such that $\sigma_1+\sigma_2+\cdots +\sigma_\ell=n$. The summands $\sigma_i$ are called \emph{parts} of the composition and  $n$ is referred to as the \emph{weight} of $\sigma$.  For example, the compositions of $4$ are
$$(\texttt{4}), \quad (\texttt{3},\texttt{1}), \quad(\texttt{1},\texttt{3}), \quad(\texttt{2},\texttt{2}), \quad (\texttt{2},\texttt{1},\texttt{1}), \quad (\texttt{1},\texttt{2},\texttt{1}), \quad (\texttt{1},\texttt{1},\texttt{2}), \quad (\texttt{1},\texttt{1},\texttt{1},\texttt{1}).$$
A \emph{palindromic} (or \emph{self-inverse}) \emph{composition} is one whose sequence of parts is the same  whether it is read from left to right or right to left.
 For example, the palindromic compositions of $4$ are
$$(\texttt{4}), \quad (\texttt{2},\texttt{2}), \quad (\texttt{1},\texttt{2},\texttt{1}), \quad  (\texttt{1},\texttt{1},\texttt{1},\texttt{1}).$$

Hoggatt and Bicknell  \cite{Hoggatt}  studied palindromic compositions having parts in a subset of positive integers. In particular,  they showed  that the total number of palindromic compositions of $n$ is given by $2^{\lfloor n/2 \rfloor}$. Moreover, the number of palindromic compositions of $n$ with parts  $\texttt{1}$ and $\texttt{2}$ are the  interleaved Fibonacci sequence $$1, \quad1, \quad 2, \quad 1, \quad 3, \quad 2, \quad 5, \quad 3, \quad8, \quad 5, \quad 13, \quad 8, \quad 21, \dots$$

The literature contains several generalizations and restrictions of the  compositions. Much of them are related to the kind of parts or summands, for example  compositions with even or odd parts \cite{Hoggatt, Mansour}, with parts in arithmetical progressions \cite{Acosta, Mun}, compositions with colored parts \cite{Aga, Coll, Shap1}, colored palindromic compositions \cite{Nar1, Guo3, Bravo2}, superdiagonal compositions \cite{Deutsch},  among other restrictions.  For further information on compositions, we refer the reader to the text by Heubach and Mansour \cite{Mansour}.

 In this paper, we study  \emph{palindromic superdiagonal  compositions}, that is a palindromic composition $(\sigma_1, \sigma_2,\ldots, \sigma_\ell)$ of $n$, with the additional condition  $\sigma_i\geq i$, for $i=1, 2, \dots, \ell$.  For example, 
 the palindromic superdiagonal compositions of $10$ are
$$(\texttt{10}), \quad (\texttt{5},\texttt{5}), \quad (\texttt{4},\texttt{2},\texttt{4}), \quad  (\texttt{3},\texttt{4},\texttt{3}).$$

 Deutsch et al. \cite{Deutsch} proved that the total number of superdiagonal compositions is given by the combinatorial sum:
 $$\sum_{k\geq 1}\binom{n-\binom k2-1}{k-1}.$$

Agarwal \cite{Aga} generalized the concept of a composition by allowing the parts to come in various colors. By a \emph{colored composition} of a positive integer $n$ we mean a composition $\sigma=(\sigma_1, \sigma_2,\dots, \sigma_\ell)$ such that the part  of size $i$ can come in one of $i$ different colors.  The colors of the summand $\texttt{i}$ are denoted by subscripts $i_1, i_2, \ldots, i_i$  for each $i \geq 1$. For example, the colored compositions of $4$  are given by
\begin{align*}
&(\texttt{4}_\texttt{1}),\quad (\texttt{4}_\texttt{2})\quad(\texttt{4}_\texttt{3}), \quad(\texttt{4}_\texttt{4}), \quad(\texttt{3}_\texttt{1}, \texttt{1}_\texttt{1}), \quad(\texttt{3}_\texttt{2}, \texttt{1}_\texttt{1}), \quad
(\texttt{3}_\texttt{3},  \texttt{1}_\texttt{1}), \quad   (\texttt{1}_\texttt{1}, \texttt{3}_\texttt{1}), \quad(\texttt{1}_\texttt{1}, \texttt{3}_\texttt{2}), \quad(\texttt{1}_\texttt{1},  \texttt{3}_\texttt{3}), \\
&(\texttt{2}_\texttt{1}, \texttt{2}_\texttt{1}), \quad (\texttt{2}_\texttt{1},  \texttt{2}_\texttt{2}), \quad(\texttt{2}_\texttt{2},  \texttt{2}_\texttt{1}), \quad(\texttt{2}_\texttt{2},  \texttt{2}_\texttt{2}), \quad(\texttt{2}_\texttt{1}, \texttt{1}_\texttt{1},  \texttt{1}_\texttt{1}), \quad(\texttt{2}_\texttt{2},  \texttt{1}_\texttt{1},  \texttt{1}_\texttt{1}), \quad (\texttt{1}_\texttt{1}, \texttt{2}_\texttt{1},  \texttt{1}_\texttt{1}), \\
&(\texttt{1}_\texttt{1},  \texttt{2}_\texttt{2},  \texttt{1}_\texttt{1}), \quad(\texttt{1}_\texttt{1},  \texttt{1}_\texttt{1},  \texttt{2}_\texttt{1}), \quad (\texttt{1}_\texttt{1}, \texttt{1}_\texttt{1},  \texttt{2}_\texttt{2}), \quad(\texttt{1}_\texttt{1}, \texttt{1}_\texttt{1},  \texttt{1}_\texttt{1}, \texttt{1}_\texttt{1}).
\end{align*}
A \emph{colored superdiagonal composition} is a colored composition such that the $i$-th part $\sigma_i$ satisfies  $\sigma_i\geq i$, for each $i\geq 1$. The colored superdiagonal compositions of $4$  are 
\begin{align}\label{ej1}
\begin{split}
&(\texttt{4}_\texttt{1}),\quad (\texttt{4}_\texttt{2}), \quad(\texttt{4}_\texttt{3}), \quad(\texttt{4}_\texttt{4}), \quad
   (\texttt{1}_\texttt{1}, \texttt{3}_\texttt{1}), \quad(\texttt{1}_\texttt{1}, \texttt{3}_\texttt{2}), \quad(\texttt{1}_\texttt{1}, \texttt{3}_\texttt{3}),\\
&(\texttt{2}_\texttt{1},\texttt{2}_\texttt{1}), \quad(\texttt{2}_\texttt{1}, \texttt{2}_\texttt{2}), \quad(\texttt{2}_\texttt{2},  \texttt{2}_\texttt{1}), \quad(\texttt{2}_\texttt{2},  \texttt{2}_\texttt{2}).
\end{split}
\end{align}

The goal of this paper is to enumerate the palindromic superdiagonal compositions and colored superdiagonal compositions. In particular we give the generating functions and  explicit combinatorial formulas involving binomial coefficients and Stirling numbers of the first kind.  

\section{Enumeration of the Palindromic Superdiagonal Compositions}

Let $\Sp$ denote the set of palindromic superdiagonal compositions.   The composition $(\sigma_1, \sigma_2,\ldots, \sigma_\ell)$ of $n$ can be represented as a \emph{bargraph} of $\ell$ columns, such that the $i$-th column contains $\sigma_i$ cells for $1\leq i \leq \ell$.  For example, in Figure \ref{Fig1} we show the superdiagonal compositions of $n=10$ with their bargraph representations. 

\begin{figure}[h]
\centering
  \includegraphics[scale=0.8]{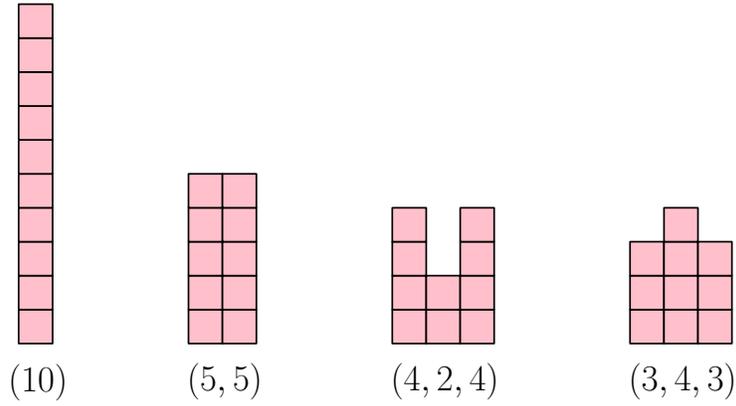}
  \caption{Palindromic compositions of $n=10$.}
  \label{Fig1}
\end{figure}

Let $\sigma$ be a composition and let us denote the weight  of  $\sigma$ by $|\sigma|$ and the number of parts of $\sigma$ by $\rho(\sigma)$. Using these parameters, we introduce this bivariate generating function
 $$S(x, y):=\sum_{\sigma \in \Sp}x^{|\sigma|}y^{\rho(\sigma)}.$$
 
In Theorem \ref{gfun1} we give an expression for the generating function $S(x,y)$.
\begin{theorem}\label{gfun1}
The bivariate generating function $S(x,y)$ is given by
  $$S(x,y)=\sum_{\sigma\in \Sp}x^{|\sigma|}y^{\rho(\sigma)}=\sum_{m\geq 0}\left(\frac{x^{3m^2+m}}{(1-x^2)^m}y^{2m} + \frac{x^{3m^2+4m+1}}{(1-x)(1-x^2)^{m}}y^{2m+1} \right).$$
\end{theorem}
\begin{proof}
Let $\sigma=(\sigma_1, \sigma_2, \dots, \sigma_{2m})$ be a  palindromic superdiagonal composition of $n$ with $2m$ parts. From the definition we have the condition $\sigma_i=\sigma_{2i}\geq 2i$, for $i=1, \dots, m$, see Figure \ref{deco1}  for a graphical representation of this case.
\begin{figure}[H]
\centering
  \includegraphics[scale=0.8]{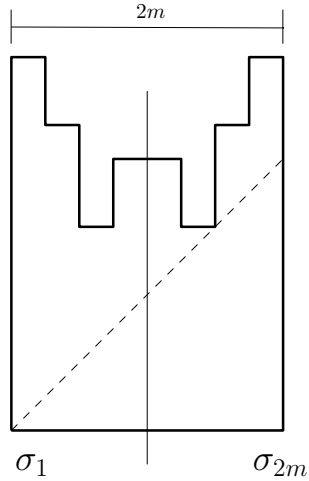}
  \caption{Decomposition of a  palindromic  superdiagonal  composition.}
  \label{deco1}
\end{figure}
The columns $i$-th and $2i$-th contribute to the generating function the term $$x^{2i}y^2+x^{2i+2}y^2+x^{2i+4}y^2+\cdots=\frac{x^{2i}y^2}{1-x^2}.$$
 Therefore the composition $\sigma$ contributes to the generating function the term $$\frac{x^{2m}y^2}{1-x^2}\frac{x^{2m+2}y^2}{1-x^2}\cdots\frac{x^{4m}y^2}{1-x^2}=\frac{x^{3m^2+m}y^{2m}}{(1-x^2)^m}.$$

If the number of parts is odd, that is $\sigma=(\sigma_1, \sigma_2, \dots, \sigma_{2m-1})$, then from a similar argument the contribution to the generating function is given by $$\frac{x^{3m^2+4m+1}}{(1-x)(1-x^2)^{m}}y^{2m+1}.$$
Summing in the above two cases over $m$ we obtain the desired result.
\end{proof}

As a series expansion, the generating function $S(x,y)$ begins with
\begin{multline*}
S(x,y)=1+x y+x^2 y+x^3 y+x^4 \left(y^2+y\right)+x^5 y+x^6 \left(y^2+y\right)+x^7 y+x^8 \left(y^3+y^2+y\right)\\+x^9 \left(y^3+y\right)+\bm{x^{10}\left(2 y^3+y^2+y\right)}+x^{11}
   \left(2 y^3+y\right)+x^{12} \left(3 y^3+y^2+y\right)+\cdots
\end{multline*}
Notice that Figure \ref{Fig1} shows the palindromic superdiagonal compositions corresponding to the  bold coefficient in the above series.
Let $s(n)$ and $s(n,k)$ denote the number of palindromic superdiagonal compositions of $n$ and the number of palindromic superdiagonal compositions of $n$ with exactly $k$ parts. Note that $s(n)=\sum_{k\geq 1}s(n,k)$. In Table \ref{tab1} we show the first few values of the sequence $s(n,k)$.

 \tiny
\begin{table}[htp]
\begin{center}
\begin{tabular}{c|cccccccccccccccccccccccccccccc}
$k\backslash n$ & 1 & 2 & 3 & 4 & 5 & 6 & 7 & 8 & 9 & 10 & 11 & 12 & 13 & 14 & 15 & 16 & 17 & 18 & 19 & 20 & 21 & 22 & 23 & 24 & 25 & 26   \\ \hline
1 & 1 & 1 & 1 & 1 & 1 & 1 & 1 & 1 & 1 & 1 & 1 & 1 & 1 & 1 & 1 & 1 & 1 & 1 & 1 & 1 & 1 & 1 & 1 & 1 & 1 & 1  \\
2& 0 & 0 & 0 & 1 & 0 & 1 & 0 & 1 & 0 & 1 & 0 & 1 & 0 & 1 & 0 & 1 & 0 & 1 & 0 & 1 & 0 & 1 & 0 & 1 & 0 & 1  \\
3& 0 & 0 & 0 & 0 & 0 & 0 & 0 & 1 & 1 & 2 & 2 & 3 & 3 & 4 & 4 & 5 & 5 & 6 & 6 & 7 & 7 & 8 & 8 & 9 & 9 & 10  \\
4& 0 & 0 & 0 & 0 & 0 & 0 & 0 & 0 & 0 & 0 & 0 & 0 & 0 & 1 & 0 & 2 & 0 & 3 & 0 & 4 & 0 & 5 & 0 & 6 & 0 & 7 \\
5& 0 & 0 & 0 & 0 & 0 & 0 & 0 & 0 & 0 & 0 & 0 & 0 & 0 & 0 & 0 & 0 & 0 & 0 & 0 & 0 & 1 & 1 & 3 & 3 & 6 & 6  \\
\end{tabular}
\end{center}
\caption{Values of $s(n,k)$, for $1\leq n \leq 26$ and $1\leq k \leq5$.}
\label{tab1}
\end{table}

\normalsize
Setting $y=1$ in Theorem \ref{gfun1} implies that the generating function for the total number of palindromic superdiagonal compositions is the following
\begin{align*}
  S(x,1)&=\sum_{n\geq 0}s(n)x^n=\sum_{m\geq 0}\frac{x^{3m^2+m}(1-x+x^{3m+1})}{(1-x)(1-x^2)^m}.
  \end{align*}
The first few values of the sequence $s(n)$  for $0\leq n \leq 28$ are 
\begin{multline*}
1, \,1, \,1, \,1, \,2, \,1, \,2, \,1, \,3, \,2, \,4, \,3, \,5, \,4, \,7, \,5, \,9, \,6, \,11, \,7, \,13, \,9, \,16, \,12, \,20, \,16, \,25, \,21, \,31 , \dots    
\end{multline*}

In Theorem \ref{teo2} we give a combinatorial expression for the sequence $s(n,k)$.

\begin{theorem}\label{teo2}
The number of palindromic superdiagonal compositions 
\begin{enumerate}
\item of $2n$ with $2k$ parts equals 
$$s(2n,2k)=\binom{n-\binom{k+1}{2}-2\binom{k}{2}-1}{k-1},$$
\item of either $2n$ or $2n-1$ with $2k-1$ parts equals 
$$s(n,2k-1)=\binom{\lfloor\frac{n-3k^2}{2}\rfloor + 2k - 1}{k-1}.$$
\end{enumerate}
\end{theorem}
\begin{proof}
From the proof of Theorem \ref{gfun1} we have
\begin{align*}
s(2n,2k)&=[x^{2n}]\frac{x^{3k^2+k}}{(1-x^2)^k}\\
&=[x^{2n-3k^2-k}]\sum_{\ell\geq 0}\binom{k+\ell-1}{k-1}x^{2\ell}\\
&=\binom{k+\frac{2n-3k^2-k}{2}-1}{k-1}\\
&=\binom{n-\binom{k+1}{2}-2\binom{k}{2}-1}{k-1}.
\end{align*}
The combinatorial formula for $s(n,2k-1)$ is obtained in a similar manner.
\end{proof}

\section{Colored Superdiagonal Compositions}
In this section we give the generating function for the total number of colored superdiagonal compositions. Remember that a composition $\sigma=(\sigma_1, \dots, \sigma_\ell)$  of $n$ is a colored superdiagonal composition if $\sigma_i\geq i$ for all $1\leq i \leq \ell$, with the additional condition that if a part is of size $i$ then it  can come in one of $i$ different colors. For example, $(\texttt{3}_\texttt{2}, \texttt{2}_\texttt{1},\texttt{5}_\texttt{3}, \texttt{5}_\texttt{2}, \texttt{6}_\texttt{6})$ is a colored superdiagonal composition of $21$.

Let $c(n)$ denote the number of colored superdiagonal compositions of $n$.  In Theorem \ref{teo2b} we give the generating for this sequence. Before we need some definitions and one lemma.

Given integers $n, k \geq 0$, let ${n \brack k}$ denote the (unsigned) \emph{Stirling numbers of the first kind},
 which are defined as connection constants in the polynomial identity
\begin{equation}\label{w1prop1}
x(x+1)\cdots(x+(n-1))=\sum_{k=0}^n{n \brack k}x^k.
\end{equation}
This sequence counts the number of permutations on $n$ elements with $k$ cycles. It is also known that this sequence satisfies the recurrence relation
\begin{align}\label{eqs}
{n \brack k}=(n-1){n-1 \brack k} + {n-1 \brack k-1},
\end{align}
with the initial conditions ${0 \brack 0}=1$ and ${n \brack 0}={0 \brack n}=0$ for $n\geq 1$.

Let $m$ be a non-negative integer.  We introduce the polynomials defined by
\begin{align}\label{def1}
Q_m(x):=\prod_{\ell=1}^m(\ell - (\ell-1)x), \quad m\geq 1,\end{align}
with the initial value $Q_0(x)=1$.
The first few polynomials are
\begin{align*}
Q_0(x)&=1, \quad Q_1(x)=1, \quad Q_2(x)=2-x,\quad Q_3(x)=2x^2-7 x+6,\\
Q_4(x)&=-6 x^3+29 x^2-46 x+24,\\
Q_5(x)&=24 x^4-146 x^3+329x^2-326 x+120,\\
Q_6(x)&=-120 x^5+874 x^4-2521 x^3+3604 x^2-2556 x+720.
\end{align*}

\begin{proposition}\label{propo1}
The polynomials $Q_m(x)$ can be expressed as
$$Q_m(x)=\sum_{k=0}^mT(m,k)x^k,$$
where  $$T(m,k)=\sum_{i=0}^m \binom{i}{m-k}{m+1 \brack m+1-i}(-1)^{m+i+k}.$$
\end{proposition}
\begin{proof}
Let $T(m,k)$ be the $k$-th coefficient of  $Q_m(x)$. From  \eqref{def1} we have 
$$Q_m(x)=mQ_{m-1}(x)-(m-1)xQ_{m-1}(x).$$
Then we obtain the recurrence relation  $T(m,k)=mT(m-1,k)-(m-1)T(m-1,k-1),$ with the initial conditions $T(0,0)=T(1,0)=1$ and $T(1,1)=0$. Let $H(m,k)=\sum_{i=0}^m \binom{i}{m-k}{m+1 \brack m+1-i}(-1)^{m+i+k}$. The sequences $T(m,k)$ and $H(m,k)$ satisfy the same recurrence relation and the same initial conditions. In fact, from \eqref{eqs} we have 
\begin{align*}
H(m,k)&=\sum_{i=0}^m \binom{i}{m-k}{m+1 \brack m+1-i}(-1)^{m+i+k}\\
&=\sum_{i=0}^m \binom{i}{m-k}\left(m{m \brack m+1-i}+{m \brack m-i}\right)(-1)^{m+i+k}\\
&=m\sum_{i=0}^{m-1} \binom{i+1}{m-k}{m \brack m-i}(-1)^{m+i+k-1}+\sum_{i=0}^m\binom{i}{m-k}{m \brack m-i}(-1)^{m+i+k}\\
&=m\sum_{i=0}^{m-1}\left(\binom{i}{m-k}+\binom{i}{m-k-1}\right){m \brack m-i}(-1)^{m+i+k-1} + H(m-1,k-1)\\
&=mH(m-1,k)-(m-1)H(m-1,k-1).
\end{align*}
Moreover, they satisfy the same initial conditions, i.e., $H(0,0)=H(1,0)=1$ and $H(1,1)=0$. 
\end{proof}

\begin{theorem}\label{teo2b}
The generating function for the number of colored superdiagonal compositions is 
$$C(x)=\sum_{m\geq 0}\frac{x^{\binom{m+1}{2}}}{(1-x)^{2m}}Q_m(x).$$
\end{theorem}
\begin{proof}
Let $\sigma=(\sigma_1,\dots,\sigma_m)$ be a non-empty colored superdiagonal  composition of $n$ with  $m$ parts.   If $\sigma_i=\ell$ with  $\ell \geq i$, then $\sigma_i$ contributes to the generating function the term $\ell x^\ell$, for $\ell\geq i$ and $1\leq i \leq m$. Let $C_m(x)$ be the generating function of the colored superdiagonal compositions  with $m$ parts. Then we have the following expression
 \begin{align*}
 C_m(x)&=\left(\sum_{i\geq 1}ix^i\right)\left(\sum_{i\geq 2}ix^i\right)\cdots \left(\sum_{i\geq m}ix^i\right)\\
&=\frac{x}{(1-x)^2}\frac{(2-x)x^2}{(1-x)^2}\cdots\frac{(m-(m-1)x)x^m}{(1-x)^2}\\
&=\frac{x^{\binom{m+1}{2}}}{(1-x)^{2m}}\prod_{\ell=1}^m(\ell-(\ell-1)x)\\
&=\frac{x^{\binom{m+1}{2}}}{(1-x)^{2m}}Q_m(x).
 \end{align*}
 Finally, summing the last expression over $m \geq 0$, we get the desired result. \end{proof}

From Theorem \ref{teo2b} and Proposition \ref{propo1} we obtain the following corollary.
\begin{corollary}
The number of colored superdiagonal compositions of $n$ is given by
$$c(n)=\sum_{m, \ell \geq 0}\binom{2m+\ell-1}{\ell}T\left(m,n-\binom{m+1}{2}-\ell\right).$$
\end{corollary}
The first few values of the  sequence  $c(n)$ are
$$1, \quad 1, \quad 2, \quad 5, \quad \textbf{11}, \quad 21, \quad 42, \quad 86, \quad 171, \quad 322, \quad 596, \dots$$ 

Notice that Equation \eqref{ej1} shows the colored superdiagonal compositions corresponding to the  bold term in the above sequence.


\begin{thebibliography}{20}


\bibitem{Acosta} J.~R.~Acosta, Y.~Caicedo, J.~P.~Poveda,  J.~L.~Ram\'irez and M.~Shattuck. Some new restricted $n$-color composition functions, \emph{J. Integer Seq.} {\bf22} (2019), Art. 19.6.4.

\bibitem{Aga} A.~K.~Agarwal. $n$-Colour compositions, \emph{Indian J. Pure Appl. Math.}  {\bf31}(11) (2000), 1421--1427.

\bibitem{Bravo2} J.~J.~Bravo, J.~L.~Herrera, J.~L.~Ram\'irez, and M.~Shattuck. $n$-Color palindromic  compositions with restricted subscripts, \emph{Proc. Indian Acad. Sci. Math. Sci.},  Accepted. 


\bibitem{Coll} A.~Collins, C.~Dedrickson, and  H.~Wang. Binary words, $n$-color compositions and bisection of the Fibonacci numbers, \emph{Fibonacci Quart.} {\bf51}(2) (2013), 130--136.

\bibitem{Deutsch} E.~Deutsch, E.~Munarini, and S.~Rinaldi. Skew Dyckpaths, area, and superdiagonal bargraphs, \emph{J. Statist. Plann. Inference} \textbf{140} (2010), 1550--1562. 


\bibitem{Guo3} Y.-H.~Guo. $n$-Color odd self-inverse compositions,  \emph{J. Integer Seq.} {\bf17} (2014), Art. 14.10.5.

\bibitem{Mansour} S.~Heubach and T.~Mansour, \emph{Combinatorics of Compositions and Words}, CRC Press, Boca Raton, FL, 2009.
\bibitem{Hoggatt}
V.~E.~Hoggatt,~Jr., and M.~Bicknell. Palindromic compositions, \emph{Fibonacci Quart.} \textbf{13} (1975), 350--356.

\bibitem{Mun} A.~Munagi. Inverse-conjugate compositions modulo $m$, \emph{J. Comb. Math. Comb. Comput. } {\bf110} (2019), 249-257.

\bibitem{Nar1} G.~Narang and A.~K.~Agarwal. $n$-Color self-inverse compositions, \emph{Proc. Indian Acad. Sci. Math. Sci.} {\bf116}(3) (2006), 257--266.



\bibitem{Shap1} C.~Shapcott. $C$-color compositions and palindromes, \emph{Fibonacci Quart.} {\bf50}(4) (2012), 297--303.






\end{thebibliography}
\end{document}